\documentclass[12pt]{amsart}
\usepackage{amsmath,amssymb,amsthm,url,amsxtra,mathrsfs}
\usepackage{fullpage}

\usepackage{hyperref}

\usepackage[all]{xy}
\usepackage{color}

\numberwithin{equation}{section}

\theoremstyle{plain}
\newtheorem{thm}[equation]{Theorem}
\newtheorem{prop}[equation]{Proposition}
\newtheorem{lem}[equation]{Lemma}
\newtheorem{cor}[equation]{Corollary}

\newtheorem{construction}[equation]{Construction}
\newtheorem*{cor*}{Corollary}
\newtheorem*{prop*}{Proposition}
\newtheorem*{thm*}{Theorem}
\newtheorem*{thma*}{Theorem A}
\newtheorem*{thmb*}{Theorem B}
\newtheorem*{thmc*}{Theorem C}

\theoremstyle{remark}

\newtheorem{exm}[equation]{Example}
\newtheorem{rmk}[equation]{Remark}

\theoremstyle{definition}
\newtheorem{defn}[equation]{Definition}

\newenvironment{enumroman}
{\begin{enumerate}}
{\end{enumerate}}

\DeclareMathOperator{\ann}{ann}

\DeclareMathOperator{\opchar}{char}

\DeclareMathOperator{\disc}{disc}

\DeclareMathOperator{\Frac}{Frac}
\DeclareMathOperator{\Gal}{Gal}
\DeclareMathOperator{\GL}{GL}
\DeclareMathOperator{\Hom}{Hom}
\DeclareMathOperator{\id}{id}

\DeclareMathOperator{\nrd}{nrd}

\DeclareMathOperator{\Pic}{Pic}

\DeclareMathOperator{\Spec}{Spec}
\DeclareMathOperator{\sq}{sq}

\DeclareMathOperator{\Sym}{Sym}
\DeclareMathOperator{\trd}{trd}

\newcommand{\A}{\mathbb A}

\newcommand{\G}{\mathbb G}

\newcommand{\NN}{\mathbb N}
\newcommand{\Q}{\mathbb Q}

\newcommand{\Z}{\mathbb Z}

\newcommand{\scrA}{\mathscr A}
\newcommand{\scrB}{\mathscr B}
\newcommand{\scrC}{\mathscr C}

\newcommand{\scrI}{\mathscr I}

\newcommand{\scrL}{\mathscr L}
\newcommand{\scrM}{\mathscr M}
\newcommand{\scrO}{\mathscr O}
\newcommand{\scrS}{\mathscr S}
\newcommand{\scrT}{\mathscr T}

\newcommand{\tbigwedge}{\textstyle{\bigwedge}}

\newcommand{\psmod}[1]{~(\textup{\text{mod}}~{#1})}

\newcommand{\la}{\langle}
\newcommand{\ra}{\rangle}

\newcommand{\Quad}{\text{\textup{Quad}}}
\newcommand{\Disc}{\text{\textup{Disc}}}
\newcommand{\AS}{\text{\textup{AS}}}

\newcommand{\scrQuad}{\text{\textup{\textbf{Quad}}}}
\newcommand{\scrDisc}{\text{\textup{\textbf{Disc}}}}

\newcommand{\scrPic}{\text{\textup{\textbf{Pic}}}}
\newcommand{\scrK}{\text{\textup{\textbf{K}}}}

\newcommand{\sbreg}{{}_{\textup{canc}}}

\newcommand{\sbuniv}{{}_{\textup{univ}}}

\newcommand{\defi}[1]{{{\fontfamily{lmss}\selectfont \textit{#1}}}}

\begin{document}

\title{Discriminants and the monoid of quadratic rings}
\author{John Voight}
\address{Department of Mathematics and Statistics, University of Vermont, 16
  Colchester Ave, Burlington, VT 05401, USA; Department of Mathematics,
  Dartmouth College, 6188 Kemeny Hall, Hanover, NH 03755, USA}
\email{jvoight@gmail.com}
\date{\today}

\begin{abstract}
We consider the natural monoid structure on the set of quadratic rings over an arbitrary base scheme and characterize this monoid in terms of discriminants.
\end{abstract}

\maketitle

Quadratic field extensions $K$ of $\Q$ are characterized by their discriminants.  Indeed, there is a bijection
\begin{align*}
\left\{ \begin{minipage}{20ex} 
\begin{center}
Separable quadratic \\[0.5ex] algebras over $\Q$ \\[0.5ex] up to isomorphism
\end{center} 
\end{minipage}
\right\} &\xrightarrow{\sim} \Q^\times/\Q^{\times 2} \\
\Q[\sqrt{d}] = \Q[x]/(x^2-d) &\mapsto d\Q^{\times 2}
\end{align*}
where a separable quadratic algebra over $\Q$ is either a quadratic field extension or the algebra $\Q[\sqrt{1}]\simeq \Q \times \Q$ of discriminant $1$.  In particular, the set of isomorphism classes of separable quadratic extensions of $\Q$ can be given the structure of an elementary abelian $2$-group, with identity element the class of $\Q \times \Q$: we have simply
\[ \Q[\sqrt{d_1}] * \Q[\sqrt{d_2}] = \Q[\sqrt{d_1d_2}] \]
up to isomorphism.  If $d_1,d_2,d_1d_2 \in \Q^\times \setminus \Q^{\times 2}$ then $\Q(\sqrt{d_1d_2})$ sits as the third quadratic subfield of the compositum $\Q(\sqrt{d_1},\sqrt{d_2})$:
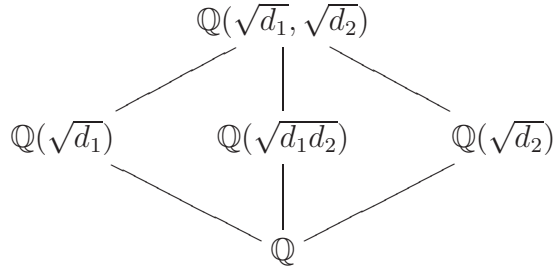
\[ 
\xymatrix{
& \Q(\sqrt{d_1},\sqrt{d_2}) \ar@{-}[dl] \ar@{-}[d] \ar@{-}[dr]  \\
\Q(\sqrt{d_1}) \ar@{-}[dr] & \ar@{-}[d] \Q(\sqrt{d_1d_2}) & \ar@{-}[dl] \Q(\sqrt{d_2}) \\
& \Q
} \] 
Indeed, if $\sigma_1$ is the nontrivial element of $\Gal(\Q(\sqrt{d_1})/\Q)$, then there is a unique extension of $\sigma_1$ to $\Q(\sqrt{d_1},\sqrt{d_2})$ leaving $\Q(\sqrt{d_2})$ fixed, similarly with $\sigma_2$, and $\Q(\sqrt{d_1d_2})$ is the fixed field of the composition $\sigma_1\sigma_2=\sigma_2\sigma_1$.  

This characterization of quadratic extensions works over any base field $F$ with $\opchar F \neq 2$ and is summarized concisely in the Kummer theory isomorphism 
\[ H^1(\Gal(\overline{F}/F),\{\pm 1\})=\Hom(\Gal(\overline{F}/F),\{\pm 1\}) \simeq F^{\times}/F^{\times 2}. \]  

On the other hand, over a field $F$ with $\opchar F = 2$, all separable quadratic extensions have trivial discriminant and instead they are  classified by the (additive) Artin-Schreier group 
\[ F/\wp(F) \quad \text{where} \quad \wp(F)=\{r+r^2 : r \in F\} \]
with the class of $a \in F$ in correspondence with the isomorphism class of the extension $F[x]/(x^2-x+a)$.  By similar considerations as above, we again find a natural structure of an elementary abelian $2$-group on the set of isomorphism classes of separable quadratic extensions of $F$.

One can extend this correspondence between quadratic extensions and discriminants integrally, as follows.  Let $R$ be a commutative ring.  A \defi{free quadratic} $R$-algebra (also called a \defi{free quadratic ring} over $R$) is an $R$-algebra $S$  (associative with $1$) that is free of rank $2$ as an $R$-module.  Let $S$ be a free quadratic $R$-algebra.  Then $S/R \simeq \wedge^2 S \simeq R$ is projective, so there is an $R$-basis $1,x$ for $S$; we find that $x^2=tx-n$ for some $t,n \in R$ and that $S$ is commutative.  The map $\sigma:S \to S$ induced by $x \mapsto t-x$ is the unique \defi{standard involution} on $S$, an $R$-linear (anti-)automorphism such that $y\sigma(y) \in R$ for all $y \in S$.  The class of the \defi{discriminant} of $S$
\[ d=d(S)=(x-\sigma(x))^2=t^2-4n \]
in $R/R^{\times 2}$ is independent of the choice of basis $1,x$.  A discriminant $d$ satisfies the congruence $d \equiv t^2 \pmod{4R}$, so for example if $R=\Z$ then $d \equiv 0,1 \pmod{4}$.

Now suppose that $R$ is an integrally closed domain of characteristic not $2$.  Then there is a bijection 
\begin{align*}
\left\{ \begin{minipage}{28ex} 
\begin{center}
Free quadratic rings over $R$ \\[0.5ex]
up to isomorphism
\end{center} 
\end{minipage}
\right\} &\xrightarrow{\sim} \{ d \in R : d \text{ is a square in $R/4R$}\}/R^{\times 2} \\
S &\mapsto d(S)
\end{align*}
For example, over $R=\Z$, the free quadratic ring $S(d)$ over $\Z$ of discriminant $d \in \Z=\Z/\Z^{\times 2}$ with $d \equiv 0,1 \psmod{4}$ is given by
\[ S(d)=\begin{cases}
\Z[x]/(x^2) \hookrightarrow \Q[x]/(x^2), & \text{ if $d=0$}; \\
\Z[x]/(x^2-\sqrt{d}x) \hookrightarrow \Q \times \Q, & \text{ if $d \neq 0$ is a square}; \\
\Z[(d+\sqrt{d})/2] \hookrightarrow \Q(\sqrt{d}), & \text{ otherwise.}
\end{cases}
\]
The set of discriminants under multiplication has the structure of a \defi{commutative monoid}, a nonempty set equipped with a commutative, binary operation and identity element.  Hence so does the set of isomorphism classes of free quadratic $R$-algebras, an operation we denote by $*$: the identity element is the class of $R \times R \simeq R[x]/(x^2-x)$.  The class of the ring $R[x]/(x^2)$ with discriminant $0$ is called an \defi{absorbing element}.

More generally, a \defi{quadratic} $R$-algebra or \defi{quadratic ring over $R$} is an $R$-algebra $S$ which is locally free of rank $2$ as an $R$-module.  By definition, a quadratic ring over $R$ localized at any prime (or maximal) ideal of $R$ is a free quadratic $R$-algebra.  Being true locally, a quadratic $R$-algebra $S$ is commutative and has a unique standard involution.

There is a natural description of quadratic $R$-algebras as a stack quotient, as follows.  A free quadratic $R$-algebra equipped with a basis $1,x$ has multiplication table uniquely determined by $t,n \in R$ and has no automorphisms, and so the functor which associates to a commutative ring $R$ the set of free quadratic $R$-algebras with basis (up to isomorphism) is represented by two-dimensional affine space $\A^2$ (over $\Z$).  The change of basis for a free quadratic $R$-algebra is of the form $x \mapsto u(x+r)$ with $u \in R^\times$ and $r \in R$, mapping 
\[ (t,n) \mapsto (u(t+2r), u^2(n+tr+r^2)). \]
Therefore, we have a map from the set of free quadratic $R$-algebras with basis to the quotient of $\A^2(R)$ by $G(R)=(\G_m \rtimes \G_a)(R)$ with the above action.  Working over $\Spec \Z$, the group scheme $G$ is naturally a subgroup scheme of $\GL_2$, but it does not act linearly on $\A^2$, and the Artin stack $[\A^2/G]$ has dimension zero over $\Spec \Z$!   Nevertheless, the set $[\A^2/G](R)$ is in bijection with the set of quadratic $R$-algebras up to isomorphism.  

Recall that a commutative $R$-algebra $S$ is \defi{separable} if $S$ is (faithfully) projective as a $S \otimes_R S$-module via the map $x \otimes y \mapsto xy$.  A free quadratic $R$-algebra $S$ is separable if and only if $d(S) \in R^\times$; so, for example, the only separable (free) $R$-algebra over $\Z$ is the ring $\Z \times \Z$ of discriminant $1$, an impoverishment indeed!  A separable quadratic $R$-algebra $S$ is \'etale over $R$, and $R$ is equal to the fixed subring of the standard involution of $S$ over $R$.  (In many contexts, one says then that $S$ is Galois over $R$ with Galois group $\Z/2\Z$, though authors differ on precise terminology.  See Lenstra \cite{Lenstra} for one approach to Galois theory for schemes.)  If $S,T$ are separable free quadratic $R$-algebras where $R$ is a Dedekind domain of characteristic not 2, having standard involutions $\sigma,\tau$, respectively, then the monoid product $S*T$ defined above (by transporting the monoid structure on the set of discriminants) is the fixed subring of $S \otimes_R T$ by $\sigma \otimes \tau$, in analogy with the case of fields.

The characterization of free quadratic $R$-algebras by their discriminants is an example of the parametrization of algebraic structures, corresponding to the Lie group $A_1$ in the language of Bhargava \cite{BhargavaGCG}.  Results in this area direction go back to Gauss's composition law for binary quadratic forms and have been extended in recent years by Bhargava \cite{BS1}, Wood \cite{MWood}, and others.  Indeed, several authors have considered the case of quadratic $R$-algebras, including Kanzaki \cite{Kanzaki} and Small \cite{Small}.  In this article, we consider a very general instance of this monoidal correspondence between quadratic $R$-algebras and discriminants over an arbitrary base scheme.  

Let $X$ be a scheme.  A \defi{quadratic $\scrO_X$-algebra} is a coherent sheaf $\scrS$ of $\scrO_X$-algebras which is locally free of rank $2$ as a sheaf of $\scrO_X$-modules.
Equivalently, a quadratic $\scrO_X$-algebra is specified by a finite locally free morphism of schemes $\phi:Y \to X$ of degree $2$ (sometimes called a \defi{double cover}): the sheaf $\phi_* \scrO_Y$ is a sheaf of $\scrO_X$-algebras that is locally free of rank $2$.  If $f:X \to Z$ is a morphism of schemes, and $\scrS$ is a quadratic $\scrO_Z$-algebra, then the pull-back $f^* \scrS$ is a quadratic $\scrO_X$-algebra.  Let $\Quad(X)$ denote the set of isomorphism classes of quadratic $\scrO_X$-algebras, and for an invertible $\scrO_X$-module $\scrL$ let $\Quad(X;\scrL) \subseteq \Quad(X)$ be the subset of those algebras $\scrS$ such that there exists an isomorphism $\tbigwedge^2 \scrS \simeq \scrL$ of $\scrO_X$-modules.  

Our first result provides an axiomatic description of the monoid structure on the set $\Quad(X)$ (Theorem \ref{thma}).  

\begin{thma*}
There is a unique system of binary operations 
\[ *_X : \Quad(X) \times \Quad(X) \to \Quad(X), \]
one for each scheme $X$, such that:
\begin{enumroman}
\item $\Quad(X)$ is a commutative monoid under $*_X$, with identity element the isomorphism class of $\scrO_X \times \scrO_X$;
\item The association $X \mapsto (\Quad(X),*_X)$ from schemes to commutative monoids is functorial in $X$: for each morphism $f:X \to Z$ of schemes, the diagram
\[
\xymatrix{
\Quad(Z) \times \Quad(Z) \ar[r]^(0.63){*_Z} \ar[d] & \Quad(Z) \ar[d]^{f^*} \\
\Quad(X) \times \Quad(X) \ar[r]^(0.63){*_X} & \Quad(X) 
} \]
is commutative; and 
\item If $X=\Spec R$ and $S,T$ are separable quadratic $R$-algebras with standard involutions $\sigma,\tau$, then $S *_{\Spec R} T$ is the fixed subring of $S \otimes_R T$ under $\sigma \otimes \tau$.
\end{enumroman}
\end{thma*}

The binary operation is defined locally (Construction \ref{keycons}): if $X=\Spec R$, and $S=R\oplus Rx$ and $T=R \oplus Ry$ are free quadratic $R$-algebras with $x^2=tx-n$ and $y^2=sy-m$ then we define the free quadratic $R$-algebra 
\[ S*T = R \oplus Rw \] 
where
\[ w^2 = (st)w - (mt^2+ns^2 - 4nm). \]
This explicit description (in the free case over an affine base) is given by Hahn \cite[Exercises 14--20, pp.\ 42--43]{Hahn}.  

A general investigation of the monoid structure on quadratic algebras goes back at least to Loos \cite{LoosMan}.  Loos gives via a universal construction a tensor product on the larger category of unital quadratic forms (quadratic forms representing $1$); this category is equivalent to the category of quadratic algebras for forms on a finitely generated, projective module of rank $2$ \cite[Proposition 1.6]{LoosMan} as long as one takes morphisms as isomorphisms in the category \cite[\S 1.4]{LoosDisc}.  (See also Loos \cite[\S 6.1]{LoosDisc} for further treatment.)  The existence of the monoid structure was also established in an unpublished letter of Deligne \cite{DeligneLetter} by a different method: he associates to every $R$-algebra its discriminant algebra (a quadratic algebra) and extends the natural operation of addition of $\Z/2\Z$-torsors from the \'etale case to the general case by geometric arguments.  Our proof of Theorem A above carries the same feel as these results, but it is accomplished in a more direct fashion and gives a characterization (in particular, uniqueness).  

Recently, there has been renewed interest in the construction of discriminant algebras (sending an $R$-algebra $A$ of rank $n$ to a quadratic $R$-algebra) by Loos \cite{LoosDisc}, Rost \cite{Rost}, and more recently by Biesel and Gioia \cite{BieselGioia}.  Indeed, Biesel and Gioia \cite[Section 8]{BieselGioia} describe the monoid operation in Theorem A over an affine base in the context of discriminant algebras.  We hope that our theorem will have some application in this context.

Our second result characterizes quadratic algebras in terms of discriminants. A \defi{discriminant} (over $X$) is a pair $(d,\scrL)$ such that $\scrL$ is an invertible $\scrO_X$-module and $d \in (\scrL\spcheck)^{\otimes 2}$ is a global section which is a \defi{square modulo $4$}: there exists a global section $\overline{t} \in \scrL\spcheck/2\scrL\spcheck = \scrL\spcheck \otimes \scrO_X/2\scrO_X$ such that $\overline{t} \otimes \overline{t}=\overline{d} \in (\scrL\spcheck)^{\otimes 2}/4(\scrL\spcheck)^{\otimes 2}$. 
We can of course also think of $d \in (\scrL\spcheck)^{\otimes 2}$ as an $\scrO_X$-module homomorphism $d:\scrL^{\otimes 2} \to \scrO_X$; but such a global section is also equivalently given by a quadratic form $D:\scrL \to \scrO_X$ (see section 2).

An isomorphism of discriminants $(d,\scrL)$, $(d',\scrL')$ is an isomorphism $f:\scrL \xrightarrow{\sim} \scrL'$ such that $(f\spcheck)^{\otimes 2}(d')=d$.  For example, if $X=\Spec R$ for $R$ a commutative ring and $\scrL=\scrO_X=\widetilde{R}$, then as above a discriminant is specified by an element $d \in R$ such that $d \equiv t^2 \pmod{4R}$ for some $t \in R$ (noting that only $t \in R/2R$ matters), and two discriminants $d,d'$ are isomorphic if and only if there exists $u \in R^\times$ such that $u^2 d'= d$.  Thinking of a discriminant as a quadratic form $D:\scrL \to \scrO_X$, its image generates a locally principal ideal sheaf $\scrI \subseteq \scrO_X$, and the set of discriminants with given locally free image $\scrI \subseteq \scrO_X$, if nonempty, is a principal homogeneous space for the group $\scrO_X^{\times}/\scrO_X^{\times 2}$.
Let $\Disc(X)$ denote the set of isomorphism classes of discriminants and $\Disc(X;\scrL) \subseteq \Disc(X)$ the subset with underlying line bundle $\scrL$.  Then the tensor product 
\[ (d,\scrL) * (d',\scrL') = (d \otimes d', \scrL \otimes \scrL') \]
gives $\Disc(X)$ and $\Disc(X;\scrO_X)$ the structure of a commutative monoid with identity element the class of $(1,\scrO_X)$.  

A quadratic $\scrO_X$-algebra $\scrS$ has a discriminant $\disc(\scrS) = (d(\scrS),\tbigwedge^2 \scrS)$, defined by
\begin{align*}
d(\scrS) : (\tbigwedge^2 \scrS)^{\otimes 2} &\to \scrO_X \\
(x \wedge y) \otimes (z \wedge w) &\mapsto (x\sigma(y)-\sigma(x)y)(z\sigma(w)-\sigma(z)w)
\end{align*}
where $\sigma$ is the unique standard involution on $\scrS$.  Although a priori the codomain of $d(\scrS)$ is $\scrS$, in fact its image lies in $\scrO_X$: indeed, if $X=\Spec R$ and $\scrS=\Spec S$ and $S$ is free with basis $1,x$, then $(\tbigwedge^2 S)^{\otimes 2}$ is free generated by $(1 \wedge x) \otimes (1 \wedge x)$ and
\[ 1 \wedge x \otimes 1 \wedge x \mapsto (x-\sigma(x))^2 \in R. \]
We have a natural forgetful map $\Disc(X) \to \Pic(X)$ where $(d,\scrL)$ maps to the isomorphism class of $\scrL$.

We say a sequence $A \xrightarrow{f} B \xrightarrow{g} C$ of commutative monoids is \defi{exact} if $f$ is injective, $g$ is surjective, and for all $z,w \in B$, we have
\[ \text{$g(z)=g(w)$ if and only if there exists $x,y \in A$ such that $xz=yw$;} \]
equivalently, the sequence is exact if and only if $f$ is injective and $g$ induces an isomorphism of monoids $B/f(A) \simeq C$.  (For a review of monoids, see Section 1.)  Similarly, a sequence $\scrA \xrightarrow{f} \scrB \xrightarrow{g} \scrC$ of sheaves of monoids is \defi{exact} if the sheaf associated to the presheaf $U \mapsto \scrB(U)/f(\scrA(U))$ is isomorphic to $\scrC$, or equivalently if the induced sequence $\scrA_x \xrightarrow{f_x} \scrB_x \xrightarrow{g_x} \scrC_x$ of monoid stalks is exact for all $x \in X$.  

We now describe the monoid $\Quad(X)$.  We begin with the statement that the forgetful map is compatible with the discriminant map, as follows.

\begin{thmb*}
Let $X$ be a scheme.  Then the diagram of commutative monoids 
\[
\xymatrix{
\Quad(X;\scrO_X) \ar[r] \ar[d]^{\disc} & \Quad(X) \ar[r]^-{\wedge^2}  \ar[d]^{\disc} & \Pic(X) \ar@{=}[d] \\
\Disc(X;\scrO_X) \ar[r] & \Disc(X) \ar[r] & \Pic(X) 
} \]
is functorial and commutative with exact rows and Zariski locally surjective columns.
\end{thmb*}

By ``Zariski locally surjective columns'', we mean that there is an (affine) open cover of $X$ where (under pullback) the columns are surjective.  (Considering the corresponding sheaves over $X$, we also obtain a surjective map of sheaves; see Theorem \ref{thmb}.)

We now turn to describe the morphism $\Quad(X;\scrO_X) \to \Disc(X;\scrO_X)$.  For this purpose, we work locally and assume $X=\Spec R$ for a commutative ring $R$; we abbreviate $\Quad(\Spec R) = \Quad(R)$ and $\Quad(\Spec R;\scrO_{\Spec R}) = \Quad(R;R)$ and similarly with discriminants.

We would like to able to fit the surjective map $\Quad(R) \xrightarrow{\disc} \Disc(R)$ of monoids into an exact sequence by identifying its kernel, but unfortunately the fibers of this map vary over the codomain.  Instead, we will describe the action of a subgroup of $\Quad(R)$ on the fibers of the map $\disc$: this is a natural generalization, as the fibers of a group homomorphism are principal homogeneous spaces for the kernel $K$ and are noncanonically isomorphic as a $K$-set to $K$ with the regular representation.

Recalling the case of quadratic extensions of a field $F$ with $\opchar F=2$, for a commutative ring $R$ we define the \defi{Artin-Schreier group} $\AS(R)$ to be the additive quotient
\[ \AS(R)=\frac{R[4]}{\wp(R)[4]} \quad \text{where} \quad \wp(R)[4]=\{n=r+r^2 \in R : r \in R\} \cap R[4] \]
and $R[4]=\{a \in R : 4a=0\}$.  We have a map $i:\AS(R) \to \Quad(R;R) \hookrightarrow \Quad(R)$ sending the class of $n \in \AS(R)$ to the isomorphism class of the algebra $S=R[x]/(x^2-x+n)$.  
The group $\AS(R)$ is an elementary abelian $2$-group since $2R[4] \subseteq \wp(R)$.

Our next main result is as follows (Theorem \ref{thmc}).

\begin{thmc*}
The fibers of the map $\disc:\Quad(R) \to \Disc(R)$ have a unique action of the group $\AS(R)$ compatible with the inclusion of monoids $\AS(R) \hookrightarrow \Quad(R)$.  Moreover, the kernel of this action on the fiber $\disc^{-1}(d R^{\times 2})$ contains $\ann_R(d)[4]$.
\end{thmc*}

Roughly speaking, Theorems B and C together say that ``a quadratic algebra is determined by its Steinitz class and its discriminant, locally up to an Artin-Schreier extension''.  These theorems could be rephrased in terms of the Grothendieck group; however, due to the existence of an absorbing element, the group $K_0(\Quad(X))$ is trivial for all schemes $X$.  

The article is organized as follows.  In section 1, we briefly review the relevant notions from monoid theory.  In section 2, we consider the monoid of discriminants; in section 3 we define the monoid of quadratic $R$-algebras and prove Theorems A and B.  In section 4 we prove Theorem C.

The author would like to thank Asher Auel, Manjul Bhargava, James Borger, and Melanie Wood for helpful suggestions.  The author is also indebted to the anonymous referee and Owen Biesel for very detailed and helpful comments and corrections.  The author was supported by an NSF CAREER Award (DMS-1151047).

\section{Monoids}

To begin, we review standard terminology for monoids.  A reference for the material in this section is Bergman \cite[Chapter 3]{Bergman}; more generally, see Burris--Sankappanavar \cite{BurrisSank} and McKenzie--McNulty--Taylor \cite{McMcT}.

A \defi{semigroup} is a nonempty set $A$ equipped with an associative binary operation 
\[ *:A \times A \to A.  \]
 A \defi{monoid} is a semigroup with identity element $1$ for $*$ (necessarily unique).  Any semigroup without $1$ can be augmented to a monoid.  Natural examples of monoids abound: the natural numbers $\mathbb{N}=\Z_{\geq 0}$ under addition, a ring $R$ under its multiplication, and the set of endomorphisms of an algebraic object (such as a variety) under composition.  A \defi{group} is a monoid equipped with an inverse map ${}^{-1}:A \to A$.

Let $A$ be a semigroup.  We say $A$ is \defi{commutative} if $xy = yx$ for all $x,y \in A$.  An \defi{absorbing element} of $A$ is an element $0 \in A$ such that $0x = x0 = 0$ for all $x \in A$; a monoid has at most one absorbing element.  Multiplicative notation for $A$ will be in general more natural for us; however, we will occasionally write $A$ additively with operation $+$, in which case the identity element is denoted $0$ and to avoid confusion $A$ will have no absorbing element.  An element $x \in A$ is (left) \defi{cancellative} if $xy=xz$ implies $y=z$ for all $y,z \in A$.

A \defi{homomorphism} of semigroups is a map $f:A \to B$ such that $f(xy)=f(x)f(y)$ for all $x,y \in A$, and a homomorphism of monoids is a homomorphism of semigroups such that $f(1_A)=1_B$.  

Let $f:A \to B$ be a homomorphism of monoids.  Unlike groups, the kernel $\ker f=\{x \in A : f(x)=1\}$ of a monoid homomorphism does not determine the structure of the image of $f$; instead, we define the \defi{kernel congruence} of $f$ by
\[ K_f=\{(x,y): f(x)=f(y)\} \subseteq A \times A. \]
The set $K_f$ defines a \defi{congruence} on $A$, an equivalence relation compatible with the operation on $A$, i.e., if $(x,y),(z,w) \in K_f$ then $(xz,yw) \in K_f$.  Conversely, given a congrence $K$ on a monoid $A$, the set operation $[x]\cdot [y] = [x\cdot y]$ on equivalence classes $[x],[y] \in A/K$ is well-defined and the quotient map $A \to A/K$ via $x \mapsto [x]$ is a surjective homomorphism of monoids with kernel $K$; any homomorphism $f:A \to B$ with $K_f \supseteq K$ factors through $A \to A/K$.  

The image $f(A)=\{f(x) : x \in A\}$ is a submonoid of $B$, but if $A,B$ are noncommutative, then not every submonoid is eligible to be the kernel of a homomorphism (just as not every subgroup is normal).  As we will be interested only in commutative monoids, and this assumption simplifies the presentation, suppose from now on that $A,B$ are commutative.  Then the set 
\[ I_f = \{(z,w) : f(x) z = f(y) w \text{ for some $x,y \in A$}\} \subseteq B \times B, \]
is a congruence called the \defi{image congruence}.  (Without the hypothesis of commutativity, $I_f$ is a relation that is reflexive and symmetric, but not necessarily transitive nor a congruence; if $A,B$ are possibly nonabelian groups, then $I_f$ is transitive and is a congruence if and only if $f(A)$ is a normal subgroup of $B$.)  Note that if $0 \in f(A)$ then $I_f=B \times B$.  We write $B/f(A) = B/I_f$.

A sequence
\begin{equation} \label{eq:exactABC}
A \xrightarrow{f} B \xrightarrow{g} C 
\end{equation}
is \defi{exact} if $f$ is injective, $g$ is surjective, and $K_g=I_f$, i.e.
\begin{center}
$g(z)=g(w)$ if and only if there exists $x,y \in A$ such that $xz=yw$;
\end{center}
equivalently, (\ref{eq:exactABC}) is exact if $f$ is injective and $g$ induces an isomorphism $B/f(A)=B/I_f \xrightarrow{\sim} C$.  A sequence of groups (\ref{eq:exactABC}) is exact as a sequence of groups if and only if it is exact as a sequence of monoids.

\begin{rmk}
We will not make use of long exact sequences of monoids here nor write the customary $0$ or $1$ at the ends of our short exact sequences.  Indeed, the straightforward extension of the notion from groups to monoids using the definitions above (kernel congruence equals image congruence) has the following defect: the sequence 
\[ \NN \xrightarrow{f} \Z \xrightarrow{j} 0 \] 
of monoids under addition has $I_f=\Z \times \Z=K_j$ even though $f$ is not surjective.  (The map $f$ is, however, an epimorphism in the category of monoids.)
\end{rmk}

We will also have use of sheaves of monoids over a scheme $X$.  A sequence $\scrA \xrightarrow{f} \scrB \xrightarrow{g} \scrC$ of sheaves of monoids is \defi{exact} if the sheaf associated to the presheaf $U \mapsto \scrB(U)/f(\scrA(U))$ is isomorphic to $\scrC$, or equivalently if the induced sequence $\scrA_x \xrightarrow{f_x} \scrB_x \xrightarrow{g_x} \scrC_x$ of monoid stalks is exact for all $x \in X$.  

Like the formation of the integers from the natural numbers, one can construct the Grothendieck group $K_0(A)$ of a commutative monoid $A$, with the universal property that for any monoid homomorphism $A \to G$ with $G$ an abelian group, there exists a unique group homomorphism $K_0(A) \to G$ such that the diagram
\[
\xymatrix{
A \ar[r] \ar[d] & G \\
K_0(A) \ar@{-->}^{\exists !}[ru]
} \]
commutes.  The group $K_0(A)$ is constructed as $A \times A$ under the equivalence relation $(x,x') \sim (y,y')$ if there exists $z \in A$ such that $xy'z = x'yz$.  Note that if $A$ has an absorbing element $0$ then $K_0(A)=\{0\}$.  The set of cancellative elements $A\sbreg$ is the largest submonoid of $A$ which can be embedded in a group, and the smallest such containing group is the Grothendieck group $K_0(A\sbreg)$.  If $\scrA$ is a sheaf of monoids, we write $\scrK_0(\scrA)$ for the sheaf of abelian groups associated to the presheaf $U \mapsto K_0(\scrA(U))$.

\section{Discriminants} \label{sec:discriminants}

In this section, we define discriminants for quadratic rings over general schemes (Definition \ref{def:defndisc}); for a discussion of discriminant modules overlapping the one presented here, see Knus \cite[\S III.3]{Knus} and Loos \cite[\S 1.2]{LoosMan}.  We also relate semi-nondegenerate quadratic forms on line bundles by their images (Lemma \ref{lem:quadformsideals}) and factor the monoid of discriminants over the Picard group (Proposition \ref{prop:propexactdiscs}).

Let $X$ be a scheme.  A \defi{quadratic form} over $X$ is a pair $(\scrM,Q)$ where $\scrM$ is a locally free $\scrO_X$-module of finite rank and $Q:\scrM \to \scrO_X$ is a \defi{quadratic map}, i.e., for all open sets $U \subseteq X$, we have
\begin{enumroman}
\item $Q(rx)=r^2Q(x)$ for all $r \in \scrO_X(U)$ and $x \in \scrM(U)$; and
\item The map $T:\scrM(U) \times \scrM(U) \to \scrO_X(U)$ defined by
\[ T(x,y) = Q(x+y)-Q(x)-Q(y) \]
is $\scrO_X(U)$-bilinear; we call $T$ the \defi{associated bilinear form}.
\end{enumroman}

An \defi{isometry} between quadratic forms $Q:\scrM \to \scrO_X$ and $Q':\scrM' \to \scrO_X$ is an $\scrO_X$-module isomorphism $f:\scrM \xrightarrow{\sim} \scrM'$ such that $Q' \circ f = Q$.  A \defi{similarity} between quadratic forms $Q,Q'$ is a commutative square:
\[
\xymatrix{ 
\scrM \ar[r]^{Q} \ar[d]^{f}_{\wr} & \scrO_X \ar[d]_g^{\wr} \\
\scrM' \ar[r]^{Q'} & \scrO_X
} \]
and so an isometry is just a similarity with $g=\id$.

%

A quadratic form $(\scrM,Q)$ is also equivalently specified by $\scrM$ and an $\scrO_X$-module homomorphism $Q:\Sym_2 \scrM \to \scrO_X$, or a global section 
\[ Q \in \Hom(\Sym_2 \scrM, \scrO_X) = (\Sym_2 \scrM)\spcheck \simeq \Sym^2 (\scrM\spcheck). \] 
(Here, $\Sym^2 \scrM$ denotes the second symmetric power of $\scrM$ and $\Sym_2 \scrM$ the submodule of symmetric second tensors of $\scrM$.)  

A quadratic form $Q:\scrM \to \scrO_X$ with associated bilinear form $T:\scrL \to \scrL \to \scrO_X$ induces a homomorphism of $\scrO_X$-modules $\scrM \to \scrM\spcheck$ defined by
\begin{align*}
\scrM(U) &\to \scrM\spcheck(U) \\
y &\mapsto (x \mapsto T(x,y))
\end{align*}
Following Knus \cite[(I.3.2)]{Knus}, we say $Q:\scrM \to \scrO_X$ is \defi{nondegenerate} if the associated map $\scrM \to \scrM\spcheck$ is injective and \defi{nonsingular} (or \defi{regular}) if the associated map $\scrM \to \scrM\spcheck$ is an isomorphism.  

Now let $\scrL$ be an invertible $\scrO_X$-module (i.e., locally free of rank $1$).  A quadratic form on $\scrL$, the case of our primary concern, is a quadratic map $Q:\scrL \to \scrO_X$, but it is given equivalently by a global section $q \in \Sym^2 (\scrL\spcheck) \simeq (\scrL\spcheck)^{\otimes 2}$.  Because will use this identification frequently, we make it explicit.  The identification is defined locally on $X$, so suppose $X=\Spec R$ and $L$ is a free module of rank $1$ over $R$.  Then $\Sym^2(L\spcheck) \simeq (L\spcheck)^{\otimes 2} \simeq (L^{\otimes 2})\spcheck$.  
Suppose $q \in \Sym^2(L\spcheck)$, so $q:L \otimes L \to R$ is an $R$-module homomorphism.  Let $L=Re$ for some $e \in L$ and define the quadratic map $Q:L \to R$ by $Q(re)=r^2 q(e \otimes e)$; this definition is independent of the choice of $e$.  Conversely, if $Q:L \to R$ is a quadratic map, then again letting $L=Re$ we define the $R$-module homomorphism $q:L \otimes L \to R$ by $q(e \otimes e) = Q(e)$.   

\begin{rmk}
One must remember the domain $\scrL$ in this identification.  Indeed, if $i:\scrT^{\otimes 2} \simeq \scrO_X$ is an isomorphism of $\scrO_X$-modules, so that $\scrT \in \Pic(X)[2]$, then $i$ defines a quadratic form $I:\scrT \to \scrO_X$, called a \defi{neutral form}, giving rise to an isomorphism 
\[ (\scrL\spcheck)^{\otimes 2} \xrightarrow{\sim}((\scrL \otimes \scrT)\spcheck)^{\otimes 2}. \]
\end{rmk}

The notions of \emph{nondegenerate} and \emph{nonsingular} can be made quite explicit for quadratic forms of rank $1$.  These conditions are local, so let $Q:L \to R$ be a quadratic form with $L=Re$.  Then $Q$ is uniquely specified by the element $Q(e)=a \in R$, and the associated bilinear form is specified by $T(e,e)=2a$.  A different choice of basis $e'=ue$ gives $Q(e')=u^2Q(e)=au^2$ with $u \in R^\times$.  We find that $Q$ is nondegenerate if and only if $2a$ is a nonzerodivisor and regular if $2a \in R^\times$, and these are well-defined.  We refine these notions (as with the half-discriminant \cite[(IV.3.1.3)]{Knus}) by saying that $Q$ is \defi{semi-nondegenerate} if $a$ is a nonzerodivisor and \defi{semi-nonsingular} (or \defi{semi-regular}) if $a$ is a unit; and we extend these notions globally to a quadratic form $Q:\scrL \to \scrO_X$ if they hold on an affine open cover.

Given two quadratic forms $Q:\scrL \to \scrO_X$ and $Q':\scrL' \to \scrO_X$, corresponding to $q \in (\scrL\spcheck)^{\otimes 2}$ and $q' \in ({\scrL'}\spcheck)^{\otimes 2}$, from the element 
\[ q \otimes q' \in (\scrL\spcheck)^{\otimes 2} \otimes ({\scrL'}\spcheck)^{\otimes 2} \simeq ((\scrL \otimes \scrL')\spcheck)^{\otimes 2} \]
we define the corresponding tensor product $Q \otimes Q' : \scrL \otimes \scrL' \to \scrO_X$: following the identification above, over $X=\Spec R$ where $L=Re$ and $L'=Re'$, then $L \otimes L' = R(e \otimes e')$ and $(Q \otimes Q')(e \otimes e')=Q(e)Q(e')$.
The tensor product gives the set of similarity classes of quadratic forms of rank $1$ over $X$ the structure of commutative monoid.

\begin{rmk}
The definition of the tensor product of quadratic forms is more subtle in general for forms of arbitrary rank; here we find the correct notion because we can think of rank $1$ quadratic forms as rank $1$ symmetric bilinear forms.
\end{rmk}

\begin{defn}
Let $Q:\scrL \to \scrO_X$ be a (rank 1) quadratic form.  We say $Q$ is \defi{cancellative} if it is cancellative in the monoidal sense, as $q \in (\scrL\spcheck)^{\otimes 2}$: if $q' \in ((\scrL')\spcheck)^{\otimes 2}$ and $q'' ((\scrL'')\spcheck)^{\otimes 2}$ have $q \otimes q'$ similar to $q \otimes q''$, then $q'$ is similar to $q''$.  

We say $Q$ is \defi{locally cancellative} if for all $x \in X$ there exists an affine open neighborhood $U \ni x$ such that $Q|_U$ is cancellative.  
\end{defn}

\begin{prop} \label{prop:loccanseminon}
A (rank 1) quadratic form $Q:\scrL \to \scrO_X$ is locally cancellative if and only if $Q$ is semi-nondegenerate.

Moreover, a locally cancellative rank 1 quadratic form $Q$ over $X$ is cancellative.  If $X$ is affine then $Q$ is cancellative if and only if it is locally cancellative.
\end{prop}

\begin{proof}
Both properties are local, so it suffices to check this over a ring $X=\Spec R$ such that the quadratic forms involved are free.  To a quadratic form $Q:L=Re \to R$, we have $Q(e)=a \in R$.  If $Q',Q''$ are similarly other rank $1$ quadratic forms with $Q(e')=a' \in R$ and $Q(e'')=a'' \in R$, then $(Q \otimes Q')(e \otimes e')=aa'$ and $(Q\otimes Q'')(e \otimes e'')=aa''$.   We have $Q' \sim Q''$ if and only if there exists $u \in R^\times$ such that $a'=ua''$, and similarly $Q \otimes Q' \sim Q \otimes Q''$.  

Thus, if $Q$ is semi-nondegenerate, then $a$ is a nonzerodivisor, so $Q \otimes Q' \sim Q \otimes Q''$ implies $aa'=uaa''$ implies $a'=ua''$ implies $Q' \sim Q''$, so $Q$ is locally cancellative.  Conversely, if $Q$ is locally cancellative and $a$ is a zero divisor, with $aa'=0$ and $a' \neq 0$, then taking $Q'(e')=a'$ and $Q''(e'')=0$ we have $Q \otimes Q' \sim Q \otimes Q''$ so $Q' \sim Q''$ and thus there exists $u \in R^\times$ such that $a'=u(0)=0$, a contradiction.  

Now for the second statement.  Let $Q:\scrL \to \scrO_X$ be locally cancellative.  By the previous paragraph, $Q$ is semi-nondegenerate.  Suppose that $Q \otimes Q' \sim Q \otimes Q''$ for rank $1$ quadratic forms $Q',Q''$; we will show that $Q' \sim Q''$.  Cancelling in $\Pic(X)$, we may assume without loss of generality that $\scrL'=\scrL''$.  Let $U=\Spec R$ be an affine open subset of $X$ in which all of $\scrL|_U=L=Re$ is free and similarly $L'=Re'=L''=Re''$.  Let $a=Q(e)$ and similarly $a'=Q'(e')$ and $a''=Q''(e'')$, so as in the previous paragraph we are given (a unique) $u \in R^\times$ such that $aa'=uaa''$.  Since $Q$ is locally cancellative, we have $a'=ua''$, and this defines a similarity $Q' \sim Q''$.  Repeating this on an open cover, the elements $u$ glue to give an element $g \in \scrO_X^\times$, and so together with the identity map on $\scrL'=\scrL''$ we therefore have a similarity $Q' \sim Q''$.

The converse in the final statement follows immediately by taking $U=X$ if $X$ is affine.
\end{proof}

The following corollary is then immediate.

\begin{cor}
The subset of locally cancellative quadratic forms over $X$ is a submonoid of the monoid of rank $1$ quadratic forms over $X$.
\end{cor} 

\begin{rmk}
The global notion of cancellative is not as robust as one may like.  Kleiman \cite{Kleiman} gives an example of a scheme $X$ and a global section $t \in \scrO_X(X)$ that is a nonzerodivisor such that it becomes a zerodivisor in an affine open $t|_U \in \scrO_X(U)$.  
\end{rmk}

The similarity class of a locally cancellative quadratic form is determined by its image (``effective Cartier divisors on a scheme are the same as invertible sheaves with a choice of regular global section'' \cite[Tag 01X0]{stacks-project}), as follows.  

\begin{lem} \label{lem:quadformsideals}
There is a (functorial) isomorphism of commutative monoids
\begin{align*}
\left\{ \begin{minipage}{36ex} 
\begin{center}
\textup{Similarity classes of rank $1$ \\[0.5ex]
locally cancellative quadratic forms \\[0.5ex]
$Q:\scrL \to \scrO_X$ \\[0.5ex]
modulo neutral forms}
\end{center} 
\end{minipage}
\right\} & \xrightarrow{\sim} \left\{ \begin{minipage}{25ex} 
\begin{center}
\textup{Locally free ideal \\[0.5ex] sheaves 
$\scrI \subseteq \scrO_X$ \\[0.5ex]
such that $[\scrI] \in 2\Pic(X)$}
\end{center} 
\end{minipage}
\right\}
\end{align*}
where the similarity class of a quadratic form $Q:\scrL \to \scrO_X$ maps to the ideal $\scrI$ of $\scrO_X$ generated by the values $Q(\scrL)$.  
\end{lem}

Ideal sheaves are a monoid under multiplication, so the ``monoid'' part of Lemma \ref{lem:quadformsideals} says that the tensor product $Q \otimes Q'$ of two quadratic forms $Q,Q'$ maps to the product $\scrI \scrI'$ of their associated ideal sheaves $\scrI,\scrI'$.

\begin{proof}
First, we show the map is well defined, which we may do locally.  Let $Q:L \to R$ be a rank $1$ locally cancellative quadratic form; then $L$ is free so we may write $L=Re$ and then $Q(L)=Q(e)R$.  By Proposition \ref{prop:loccanseminon}, we know that $Q(e)$ is a nonzerodivisor, and this is independent of the choice of $e$ (up to a unit of $R$).  If $Q':L' \to R$ is similar to $Q$, then there exist $R$-linear isomorphisms $f:L \to L'$ and $g:R \to R$ such that $Q'(f(x))=g(Q(x))$ for all $x \in L$.  Letting $e'=f(e)$ we have $L'=Re'$.  The map $g$ must be of the form $g(a)=ua$ for some $u \in R^\times$, so $Q'(L')=Q'(f(L))=uQ(L)=u Q(e) R = Q(e)R$, and so the image is a well-defined principal ideal.  

Now let $Q:\scrL \to \scrO_X$ be a locally cancellative quadratic form of rank $1$, corresponding to the global section $q \in (\scrL^{\otimes 2})\spcheck$.  We claim that $q:\scrL^{\otimes 2} \to \scrO_X$ is an isomorphism onto its image $\scrI=Q(\scrL)=q(\scrL^{\otimes 2})$.  If $U=\Spec R$ is an affine open set where $\scrL|_U=Re$, then $q(L^{\otimes 2})= q(e \otimes e)R$; if further $U$ is such that $q$ is cancellative over $U$, we have $q(e \otimes e)$ is a nonzerodivisor, so $q$ is injective on $U$ (and $q(L)$ is free).  By hypothesis, such affine opens $U$ cover $X$, so we have $[\scrI]=[\scrL^{\otimes 2}] \in \Pic(X)$.  

Next, let $q:\scrL \to \scrO_X$ and $q': \scrL' \to \scrO_X$ be locally cancellative quadratic forms such that $q(\scrL)=q'(\scrL')=\scrI$.  Since $[\scrI]=[\scrL^{\otimes 2}]=[(\scrL')^{\otimes 2}]$, tensoring $q'$ by a neutral form we may assume that $f:\scrL \xrightarrow{\sim} \scrL'$.  Then on any open affine $U=\Spec R \subseteq X$, where $\scrL|_U=Re$ and $\scrL'|_U=Re'$, we have $q(L)=q(e)R=q'(e')R$.  Therefore, there exists $u \in R$ such that $q(e)=u q'(e')$ and $u' \in R$ such that $q'(e')=u' q(e)$.  Thus $q(e)(1-uu')=0$.  On an open affine $U$ where $q$ is cancellative, we have  $uu'=1$, so $u \in R^\times$.  Moreover, the element $u$ is unique, since if $q(e)=u q'(e') = v q'(e')$ then $(u-v)q'(e')=0$ so since $q'$ is cancellative, we have $u=v$.  Repeating this argument on an open cover where both $\scrL$ and $\scrL'$ are free, there exists (a unique) $u \in \scrO_X^\times$ giving rise to an isomorphism $\scrO_X \xrightarrow{\sim} \scrO_X$ such that $q'f = uq$, so that $q,q'$ are similar.  

Finally, the map is surjective.  We are given that there exists an invertible bundle $\scrL$ such that $\scrL^{\otimes 2} \simeq \scrI$. The embedding $\scrL^{\otimes 2} \simeq \scrI \hookrightarrow \scrO_X$ then defines a locally cancellative quadratic form $Q:\scrL \to \scrO_X$ with values $Q(\scrL)=\scrI$, as can be readily checked locally.
\end{proof}

\begin{rmk}
A (locally) cancellative rank $1$ quadratic form might not pull back to a (locally) cancellative form under an arbitrary morphism of schemes.
\end{rmk}

To work with discriminants, we will work modulo $2$ and $4$ as follows.  The multiplication by $4$ map on $\scrO_X$ gives a closed immersion 
\[ X_{[4]} = X \times_{\Spec \Z} \Spec \Z/4\Z \hookrightarrow X \] 
and the pullback $\scrL_{[4]} = \scrL \otimes \scrO_X/4\scrO_X$ is an invertible $\scrO_{X_{[4]}}$-module, equipped with a map ${}_{[4]}:\scrL \to \scrL_{[4]}$.  We can also further work modulo $2$, obtaining $\scrL_{[2]}$.  

Let $R$ be a commutative ring.  Then squaring gives a well-defined map of sets
\begin{equation} \label{eqn:R2R}
\begin{aligned}
\sq:R/2R &\to R/4R \\
\sq(t+2R) &= t^2 + 4R
\end{aligned}
\end{equation}
The map $\sq$ is functorial in $R$ and canonically defined, so we can sheafify: if $\scrL$ is an invertible $\scrO_X$-module, there is a unique map
\[ \sq:\scrL_{[2]} \to (\scrL_{[4]})^{\otimes 2} \]
locally defined by \eqref{eqn:R2R}.  Explicitly, for an affine open $U=\Spec R$ of $X$ where $\scrL(U)=L=Re$, we may write $\scrL_{[2]}(U) = (R/2R)e$ and $\scrL_{[4]}^{\otimes 2}(U) = (R/4R)(e \otimes e)$, and
\[ \sq((t+2R)e) = \sq(t+2R)(e \otimes e) \] 
is well-defined (independent of the choice of $e$).

\begin{defn} \label{def:defndisc}
A \defi{discriminant} over $X$ is a pair $(d,\scrL)$ where $\scrL$ is an invertible $\scrO_X$-module and $d \in (\scrL\spcheck)^{\otimes 2}$ such that there exists $t \in \scrL_{[2]}\spcheck$ with 
\begin{equation} \label{eqn:sqcond}
\sq(t)=d_{[4]} \in (\scrL_{[4]}\spcheck)^{\otimes 2}. 
\end{equation}
\end{defn}

If $2$ is invertible on $X$, then $X_{[4]}$ is the empty scheme and the square condition \eqref{eqn:sqcond} is vacuously satisfied.  

\begin{defn}
An \defi{isomorphism} between discriminants $(d,\scrL)$ and $(d',\scrL')$ is an isomorphism $f:\scrL \xrightarrow{\sim} \scrL'$ such that $(f\spcheck)^{\otimes 2}:(\scrL'{}\spcheck)^{\otimes 2} \to (\scrL\spcheck)^{\otimes 2}$ has $(f\spcheck)^{\otimes 2}(d')=d$.  
\end{defn}

Equivalently, an isomorphism between discriminants is an isometry (not a similarity!)\  between the corresponding quadratic forms.  

In what follows, we will often write abbreviate $d$ for a discriminant $(d,\scrL)$, and refer to $\scrL$ as the underlying invertible sheaf.

Let $\Disc(X)$ denote the set of discriminants over $X$ up to isomorphism.  For an invertible sheaf $\scrL$ on $X$, let $\Disc(X;\scrL) \subseteq \Disc(X)$ denote the subset of isomorphism classes of discriminants $d$ whose underlying invertible sheaf is (isomorphic to) $\scrL$.  Define $\scrDisc(X)$ to be the sheaf associated to the presheaf $U \mapsto \Disc(U)$.

\begin{lem} \label{lem:orbitsdisc}
There is a functorial bijection between 
\[ \Disc(X;\scrO_X) \xrightarrow{\sim} \{d \in \scrO_X : \text{$d$ is a square modulo $4\scrO_X(X)$}\}/\scrO_X(X)^{\times 2}. \]
\end{lem}

\begin{proof}
Suppose that $f$ is an isomorphism between discriminants $d,d'$.  Let $U=\Spec R \subseteq X$ be an affine open subset.  Then $d|_U:R \to R$ is a quadratic map with $d(r)=r^2 d(1)$ for all $r \in R$.  The restriction $f|_U:\scrO_X(U) = R \to R$ is an isomorphism and so is identified with a unique element $u \in R^\times$, and so in $R$ we have $d'|_U(f|_U(1)) = d'|_U(u)= u^2 d'|_U(1)=d|_U(1)$; 
by gluing, there exists a (unique) global section $u \in \scrO_X(X)^\times$ such that $d=u^2 d'$.
\end{proof}

\begin{exm} \label{exm:orbitsdisc}
If $X=\Spec R$ where $R$ is a PID or local ring, then by Lemma \ref{lem:orbitsdisc}, $\Disc(R)=\Disc(R;R)$ is canonically identified with
\[ \{d \in R : d \text{ is a square in $R/4R$}\}/R^{\times 2}. \]
So for $R=\Z$, since $\Z^{\times 2}=\{1\}$ we recover the usual set of discriminants as those integers $d \equiv 0,1 \pmod{4}$.  
\end{exm}

\begin{lem}
$\Disc(X)$ has the structure of commutative monoid under tensor product, with identity element represented by the class of $(1,\scrO_X)$.  Moreover, $\Disc(X;\scrO_X)$ is a submonoid of $\Disc(X)$ with absorbing element $(0,\scrO_X)$.  
\end{lem}

\begin{proof}
The binary operation of tensor product on quadratic forms restricts to a binary operation on discriminants: if $(d,\scrL)$ and $(d',\scrL')$ are discriminants, with $t \in \scrL_{[2]}\spcheck$ satisfying $\sq(t)=d_{[4]} \in (\scrL_{[4]}\spcheck)^{\otimes 2}$ and similarly for $(d',\scrL')$, then
\[ \sq(t \otimes t') = d_{[4]} \otimes d'_{[4]} = (d \otimes d')_{[4]} \in (\scrL_{[4]}\spcheck)^{\otimes 2} \otimes ((\scrL_{[4]}'{})\spcheck)^{\otimes 2} \simeq ((\scrL \otimes \scrL')\spcheck_{[4]})^{\otimes 2}. \]  
This definition is independent of the choice of a representative discriminant in an isomorphism class, so we obtain a binary operation on $\Disc(X)$.  This operation is associative and commutative and $(1,\scrO_X)$ is an identity by definition of the tensor product.  The subset $\Disc(X;\scrO_X)$ is closed under tensor product, and $(0,\scrO_X)$ is visibly an absorbing element.  
\end{proof}

To a discriminant $(d,\scrL)$, we can forget the quadratic map $d$ and consider only the isomorphism class of the $\scrO_X$-module $\scrL$: this gives a map
\[ p:\Disc(X) \to \Pic(X). \]

\begin{prop} \label{prop:propexactdiscs}
The sequence 
\[ \Disc(X;\scrO_X) \to \Disc(X) \xrightarrow{p} \Pic(X) \]
of commutative monoids is exact.
\end{prop}

\begin{proof}
The map $p:\Disc(X) \to \Pic(X)$ is surjective because an invertible module $\scrL$ has the zero quadratic form $d=0$, which is a discriminant taking $t=0$.  (One can hardly do better in general, since it may be the case that $\scrL(X)=\{0\}$ has no nonzero global sections.) 

Let $i:\Disc(X;\scrO_X) \hookrightarrow \Disc(X)$ be the inclusion map.  We show that $I_i = K_p$.  The inclusion $I_i \subseteq K_p$ is easy, so we show the reverse inclusion.  Let $d,d'$ be discriminants and suppose $([d],[d']) \in K_p$; then the underlying invertible sheaves of $d,d'$ are isomorphic, so we may assume without loss of generality that $d,d' \in (\scrL\spcheck)^{\otimes 2}$.  To show $([d],[d']) \in I_i$, we need to show that there exist $\delta,\delta' \in \Disc(X;\scrO_X)$ such that $\delta \otimes d' = \delta' \otimes d$.  For this purpose, we may take $\delta=\delta'=0$.  More generally, we could take any $\delta \in (d : d') = \{ \delta \in \scrO_X(X) : \delta d' \in \la d \ra \} \subseteq \scrO_X(X)$.
\end{proof}

\section{Quadratic algebras}

In this section, we give a monoid structure on the set of isomorphism classes of quadratic algebras.  We begin by discussing the algebras over commutative rings, then work over a general base scheme.  For more on quadratic rings and standard involutions, see Knus \cite[Chapter I, \S 1.3]{Knus} and Voight \cite[\S 1--2]{VoightLowRank}.

Let $R$ be a commutative ring.  An \defi{$R$-algebra} is an associative ring $B$ with $1$ equipped with an embedding $R \hookrightarrow B$ of rings (mapping $1 \in R$ to $1 \in B$) whose image lies in the center of $B$; we identify $R$ with this image in $B$.  A homomorphism of $R$-algebras is required to preserve $1$.  

\begin{defn}
A \defi{quadratic} $R$-algebra (or \defi{quadratic ring} over $R$) is an $R$-algebra $S$ that is locally free of rank $2$ as an $R$-module.  
\end{defn}

Let $S$ be a quadratic $R$-algebra.  Then $S$ is commutative, and there is a unique \defi{standard involution} on $S$, an $R$-linear homomorphism $\sigma: S \to S^{\textup{op}}=S$ such that $\sigma(\sigma(x))=x$ and $x\sigma(x) \in R$ for all $x \in S$.  Consequently, we have a linear map $\trd:S \to R$ defined by $\trd(x) = x + \sigma(x)$ and a multiplicative map $\nrd:S \to R$ by $\nrd(x)=x\sigma(x)=\sigma(x)x$ with the property that $x^2-\trd(x)x+\nrd(x)=0$ for all $x \in S$.  

If $S$ is free over $R$ with basis $1,x$, then the multiplication table is determined by the multiplication $x^2=tx-n$: consequently, we have a bijection
\begin{equation}  \label{eq:freequadtn}
\begin{aligned} 
\left\{ \begin{minipage}{26.5ex} 
\begin{center}
Free quadratic $R$-algebras \\[0.5ex]
$S$ over $R$ equipped \\[0.5ex] 
with a basis $1,x$ 
\end{center} 
\end{minipage}
\right\} &\xrightarrow{\sim} R^2 \\
S &\mapsto (\trd(x),\nrd(x))=(x+\overline{x},x\overline{x})=(t,n).
\end{aligned}
\end{equation}
A change of basis for a free quadratic $R$-algebra is of the form $x \mapsto u(x+r)$ with $u \in R^\times$ and $a \in R$, mapping 
\begin{equation} \label{eq:changeofbasis}
(t,n) \mapsto (u(t+2r), u^2(n+tr+r^2)).
\end{equation}

We have identified $R \subseteq S$ as a subring; the quotient $S/R$ is locally free of rank $1$.  Therefore, we have a canonical identification 
\begin{equation} \label{eqn:tSR}
\begin{aligned}
S/R &\xrightarrow{\sim} \tbigwedge^2 S \\
x + R &\mapsto 1 \wedge x.
\end{aligned}
\end{equation}

\begin{lem} \label{lem:isdisc}
Let $S$ be a quadratic $R$-algebra.  Then the map
\begin{equation} \label{eqn:discmap}
\begin{aligned}
d: (\tbigwedge^2 S)^{\otimes 2} &\to R \\
(x \wedge y) \otimes (z \wedge w) &\mapsto (x\sigma(y)-\sigma(x)y)(z\sigma(w)-\sigma(z)w).
\end{aligned}
\end{equation}
is a discriminant.
\end{lem}

We have
\begin{equation} \label{eq:whenfree}
\begin{aligned}
d( (1 \wedge x)^{\otimes 2}) &= (x-\sigma(x))^2=(2x-\trd(x))^2 \\
&= 4x^2-4x\trd(x)+\trd(x)^2 =\trd(x)^2-4\nrd(x)
\end{aligned}
\end{equation}
in the lemma, as one might expect.  We accordingly call the quadratic map $d=d(S)$ in Lemma \ref{lem:isdisc} the \defi{discriminant} of $S$.

\begin{proof}
We define the map
\begin{align*}
t:\tbigwedge^2(S/2S) &\to R/2R \\
t(1 \wedge x)&=\trd(x)
\end{align*} via the identification \eqref{eqn:tSR}.  The map $t$ is well-defined since $\trd(x+r)=\trd(x)+2r \equiv \trd(x) \pmod{2R}$.  We then verify that 
\[ \sq(t)((1\wedge x)^{\otimes 2})=t(1\wedge x)^2=\trd(x)^2\equiv \trd(x)^2-4\nrd(x) = d((1 \wedge x)^{\otimes 2}) \pmod{4R} \]
by \eqref{eq:whenfree}.
\end{proof}

Recall that a commutative $R$-algebra $B$ is \defi{separable} if $B$ is projective as a $B \otimes_R B$-module via the map $x \otimes y \mapsto xy$.  If $B \simeq R[x]/(f(x))$ with $f(x) \in R[x]$, then $B$ is separable if and only if the ideal generated by $f(x)$ and its derivative $f'(x)$ is the unit ideal.  

\begin{lem}
A quadratic $R$-algebra $S$ is separable if and only if its discriminant $d$ is an isomorphism.
\end{lem}

\begin{proof}
The map $d:(\tbigwedge^2 S)^{\otimes 2} \to R$ is an isomorphism if and only if it is locally an isomorphism, so we reduce to the case where $S=R[x]/(x^2-tx+n)=R[x]/(f(x))$.  Then by (\ref{eq:whenfree}), $d$ is an isomorphism if and only if $\trd(x)^2-4\nrd(x) \in R^\times$ if and only if $S$ is separable \cite[Chapter I, (7.3.4)]{Knus}.
\end{proof}

\begin{cor}
If $S$ is a separable quadratic $R$-algebra over $R$ then $\tbigwedge^2 S \in \Pic(R)[2]$.
\end{cor}

Now let $X$ be a scheme.  

\begin{defn}
A \defi{quadratic $\scrO_X$-algebra} is a sheaf $\scrS$ of $\scrO_X$-algebras that is locally free of rank $2$ as a sheaf of $\scrO_X$-modules: there is a basis of open sets $U$ of $X$ such that $\scrS(U)$ is free of rank $2$ as an $\scrO_X(U)$-module.  
\end{defn}

Equivalently, a quadratic $\scrO_X$-algebra is given by a \defi{double cover} $\phi:Y \to X$, a finite locally free morphism of schemes of degree $2$: the sheaf $\phi_* \scrO_Y$ is a sheaf of $\scrO_X$-algebras that is locally free of rank $2$.  By uniqueness of the standard involution, we obtain a \defi{standard involution} on $\scrS$, a standard involution on $\scrS(U)$ for all open sets $U$ (covering each by affine opens where $\scrS$ is free), and in particular maps $\trd$ and $\nrd$ on $\scrS$.  

Analogous to Lemma \ref{lem:isdisc}, we have the following result.

\begin{prop} \label{prop:isdiscnonfree}
Let $\scrS$ be a quadratic $\scrO_X$-algebra.  Then there is a unique discriminant $d:\tbigwedge^2 \scrS \to \scrO_X$ that coincides locally with the one defined by \eqref{eqn:discmap}.
\end{prop}

\begin{proof}
Let $\scrL=\tbigwedge^2 \scrS$.  We must exhibit $t \in \scrL_{[2]}\spcheck$ such that $\sq(t) = d_{[4]} \in (\scrL_{[4]}\spcheck)^{\otimes 2}$, where ${}_{[4]}$ denotes working modulo $4$, as in the previous section.  We adapt the argument in Lemma \ref{lem:isdisc} to a global setting.  Working first locally, let $U=\Spec R \subseteq X$ be an open set where $\scrS(U)=S$ and $\scrL(U)=L=\tbigwedge^2 S$.  Since $\trd(x+r)=\trd(x)+2r$ for $r \in R$ and $x \in S$, the map
\begin{align*}
t : \tbigwedge^2 (S/2S) \simeq L/2L &\to R/2R \\
x \wedge y &\mapsto \trd(x+y)
\end{align*}
is well-defined (since $t(x \wedge x)=\trd(2x)=0$) and $R$-linear.  This map does not depend on any choices, so repeating this on an open cover, we obtain an element $t \in \scrL_{[2]}\spcheck$.  

Now we verify that $\sq(t) = d_{[4]} \in (\scrL_{[4]}\spcheck)^{\otimes 2}$.  We may do so on an open cover, where $\scrS$ is free, so let $S=R \oplus R x$, and $\tbigwedge^2 S=L=R(1 \wedge x)^{\otimes 2}$.  Then 
\[ d((1 \wedge x)^{\otimes 2}) = \trd(x)^2-4\nrd(x) \equiv \trd(x)^2 \pmod{4R}. \] 
On the other hand, by definition we have 
\[ \sq(t)((1 \wedge x)^{\otimes 2}) \equiv t(1 \wedge x)^2 = \trd(1+x)^2 = (2+\trd(x))^2 \equiv \trd(x)^2 \pmod{4R}. \]
The result follows.
\end{proof}

A quadratic $\scrO_X$-algebra $\scrS$ is separable if and only if $d$ induces an isomorphism of $\scrO_X$-modules $(\tbigwedge^2 \scrS)^{\otimes 2} \xrightarrow{\sim} \scrO_X$, as this is true on any affine open set.  

Let $\Quad(X)$ denote the set of isomorphism classes of quadratic $\scrO_X$-algebras, and for an invertible $\scrO_X$-module $\scrL$ let $\Quad(X;\scrL) \subseteq \Quad(X)$ be the subset of those algebras $\scrS$ such that there exists an isomorphism $\tbigwedge^2 \scrS \simeq \scrL$ of $\scrO_X$-modules.  Define similarly $\scrQuad(X)$ to be the sheaf associated to the presheaf $U \mapsto \Quad(U)$.

We now give $\Quad(X)$ the structure of a commutative monoid.  

\begin{construction} \label{keycons}
Let $X=\Spec R$ and let $S=R\oplus Rx$ and $T=R \oplus Ry$ be free quadratic $R$-algebras with choice of basis.  Let $x^2=tx-n$ and $y^2=sy-m$ so
\[ t=\trd(x), n=\nrd(x), s=\trd(y), m=\nrd(y) \quad\text{ with } t,n,s,m \in R. \]  
Then we define the free quadratic $R$-algebra 
\[ S*T = R \oplus Rw \] 
where
\begin{equation} \label{eq:keydefn}
w^2 = (st)w - (mt^2+ns^2 - 4nm).
\end{equation}
\end{construction}

Construction \ref{keycons} has been known for some time, e.g., it is given by Hahn \cite[Exercises 14--20, pp.\ 42--43]{Hahn}.  (See the introduction for further context and references.)

\begin{lem} \label{lem:check_monoid}
Construction \/\textup{\ref{keycons}} is functorial with respect to the base ring $R$.  The operation $*$ gives the set of free quadratic $R$-algebras with basis the structure of commutative monoid with identity element $R \times R = R[x]/(x^2-x)$ and  absorbing element $R[x]/(x^2)$.
\end{lem}

\begin{proof}
Functoriality is clear, and $S*T=T*S$ for all free quadratic $R$-algebras $S,T$ by the symmetry of the construction.  It is routine to check associativity.  To check that $E=R \times R$ is the identity element for $*$ we simply substitute $s=1,m=0$ to obtain $S*E=S$; and a similar check works for the absorbing element.
\end{proof}

\begin{rmk}
Construction \ref{keycons} generalizes the Kummer map, presented in the introduction.  Indeed, suppose that $R$ is a PID or local ring and $2 \in R^\times$.  Then by completing the square, any quadratic $R$-algebra $S$ is of the form $S=R[x]/(x^2-n)=R[\sqrt{n}]$ where $n=d(S)/4$.  So if $S=R[\sqrt{n}]$ and $T=R[\sqrt{m}]$, then $S * T = R[x]/(x^2-4nm) \simeq R[\sqrt{nm}]$.  

At the same time, Construction \ref{keycons} generalizes Artin-Schreier extensions of fields.  Suppose that $R=k$ is a field of characteristic $2$.  Then every separable extension of $k$ can be written in the form $k[x]/(x^2-x+n)$, and 
\[ k[x]/(x^2-x+n) * k[x]/(x^2-x+m) = k[x]/(x^2-x+(m+n)). \] 
Since $4=0$, the discriminant of every such algebra has class $1$ in $R/R^{\times 2}$.

In the above construction, if $S,T$ are separable over $R$, so that they are (\'etale) Galois \cite{Lenstra} extensions of $R$ (with the standard involutions $\sigma,\tau$ respectively as the nontrivial $R$-algebra automorphism), then the algebra $S * T$ is the subalgebra of $S \otimes_R T$ fixed by the product of the involutions $\sigma \otimes \tau$ acting on $S \otimes_R T$ \cite[Proposition 1]{Small}.  

In all cases, a direct calculation shows that Equation \ref{eq:keydefn} is satisfied by the element 
\[ w = x\otimes y + \sigma(x) \otimes \tau(y) \in S \otimes_R T; \]
this will figure in the proof of Theorem \ref{thma}.  However, there is no reason why the $R$-algebra generated by $w$ need be free of rank $2$ over $R$; for example, if $R$ has characteristic $2$ and $\sigma(x)=x$, $\tau(y)=y$, then $w=0$.  Thus, Construction \ref{keycons} can be thought of as a formal way to create a fixed subalgebra of $S \otimes_R T$ under the involution given by the product of standard involutions.  
\end{rmk}

\begin{lem} \label{lem:funct_isom}
Construction \/\textup{\ref{keycons}} is functorial with respect to isomorphisms: if 
\begin{align*}
\phi:S=R \oplus Rx &\xrightarrow{\sim} S'=R \oplus Rx' \\ 
\psi:T=R \oplus Ry &\xrightarrow{\sim} T'=R \oplus Ry'
\end{align*} 
are $R$-algebra isomorphisms of quadratic $R$-algebras, then there is a canonical isomorphism 
\[ \phi*\psi:S*T \xrightarrow{\sim} S'*T'. \]
\end{lem}

\begin{proof}
There exist unique $u,v \in R^\times$ and $r,q \in R$ such that
$\phi(x)=ux'+r$ and $\psi(y)=vy'+q$.  Because $\phi$ is an $R$-algebra homomorphism, both $\phi(x)$ and $x$ satisfy the same unique monic quadratic polynomial, and from
\[ (ux'+r)^2=t(ux'+r)-n \] 
we conclude that
\[ (x')^2 = u^{-1}(t-2r)x' - u^{-2}(n-tr-r^2) = t'x - n' \]
so $t=ut'+2r$ and $n=u^2n'+tr+r^2$.  Similarly, we obtain $s=vs'+2q$ and $m=v^2m'+sq+q^2$.  We claim then that the map
\begin{align*}
\phi*\psi:S*T &\xrightarrow{\sim} S'*T' \\
(\phi*\psi)(w)&=(uv)w'+(qt+rs-2qr)
\end{align*}
is an isomorphism; for this we simply verify that 
\[ ((uv)w'+(qt+rs-2qr))^2 = st((uv)w'+(qt+rs-2qr))-(mt^2+ns^2-4nm) \]
and the result follows.
\end{proof}

\begin{lem} \label{lem:ref_glue}
Let $\scrS,\scrT$ be quadratic $\scrO_X$-algebras.  Then there is a unique quadratic $\scrO_X$-algebra $\scrS * \scrT$ up to $\scrO_X$-algebra isomorphism with the property that on any affine open set $U \subseteq X$ such that $S=\scrS(U)$ and $T=\scrT(U)$ are free, we have 
\[ (\scrS * \scrT)(U) \simeq S*T \]
as in Construction \textup{\ref{keycons}}.
\end{lem}

\begin{proof}
This proposition is a standard application of gluing; we give the argument for completeness.  Let $\{U_i=\Spec R_i\}$ be an open affine cover of $X$ on which 
\begin{center}
$\scrS(U_i)=S_i=R_i \oplus R_ix_i$ and $\scrT(U_i)=T_i=R_i \oplus R_iy_i$ 
\end{center}
are free.  We define $(\scrS*\scrT)(U_i)=S_i*T_i = R_i \oplus R_i w_i$ according to Construction \ref{keycons}.  We glue these according to the isomorphisms on $\scrS$ and $\scrT$ using Lemma \ref{lem:funct_isom}, as follows.  We have $U_i \cap U_j=U_j \cap U_i = \bigcup_k U_{ijk}$ covered by open sets $U_{ijk}=\Spec R_{ik} \simeq \Spec R_{jk}$ distinguished in $U_i$ and $U_j$.  Because $\scrS$ is a sheaf, we have compatible isomorphisms
\[ \phi_{ijk}: R_{ik} \oplus R_{ik} x_i= \scrS(\Spec R_{ik}) \simeq \scrS(\Spec R_{jk}) =  R_{jk} \oplus R_{jk} x_j  \]
for each such open set.  Similarly, we obtain compatible isomorphisms $\psi_{ijk}$ for $\scrT$ over the same open cover.  By Lemma \ref{lem:funct_isom}, we obtain compatible isomorphisms 
\[ \phi_{ijk} * \psi_{ijk} : (\scrS*\scrT)(\Spec R_{ik}) \simeq (\scrS*\scrT)(\Spec R_{jk}) \]
and can thereby glue on $X$ to obtain a quadratic $\scrO_X$-algebra, unique up to $\scrO_X$-algebra isomorphism.
\end{proof}

\begin{cor}
Construction \textup{\ref{keycons}} gives $\Quad(X)$ the structure of a commutative monoid, functorial in $X$, with identity element the isomorphism class of $\scrO_X \times \scrO_X$.
\end{cor}

\begin{proof}
Lemma \ref{lem:ref_glue} shows that Construction \ref{keycons} extends to $X$ and is well defined on the set of isomorphism classes $\Quad(X)$ of quadratic $\scrO_X$-algebras.  To check that we obtain a functorial commutative monoid, it is enough to show this when $X$ is affine, and this follows from Lemmas \ref{lem:check_monoid} and \ref{lem:funct_isom}.
\end{proof}

\begin{lem}
If $\scrS$ is a separable quadratic $\scrO_X$-algebra, then $\scrS * \scrS \simeq \scrO_X \times \scrO_X$.
\end{lem}

\begin{proof}
By gluing, it is enough to show this on an affine cover.  Suppose $S=R[x]/(x^2-tx+n)$ has discriminant $d=t^2-4n$.  Then by definition, we have 
\[ S*S = R[w]/(w^2-t^2w+2n(t^2-2n)); \]
 with the substitution $w \leftarrow w-2n$, we find that $S*S \simeq R[w]/(w^2-dw)$.  Since $d \in R^\times$, the replacement $w \leftarrow wd^{-1}$ yields an isomorphism $S*S \simeq R \times R$.
\end{proof}

\begin{lem} \label{lem:homquad}
If $\scrS,\scrT \in \Quad(X)$ then
\[ d(\scrS*\scrT) = d(\scrS)d(\scrT) \in \Disc(X) \]
and
\[ \tbigwedge^2 (\scrS * \scrT) \simeq \tbigwedge^2 \scrS \otimes \tbigwedge^2 \scrT. \]
\end{lem}

\begin{proof}
If $S=\scrS(U)$ and $T=\scrT(U)$ are as in Construction \ref{keycons}, then
\begin{equation} \label{eq:disccalc}
\begin{aligned}
d(S * T)((1 \wedge (x \otimes y))^{\otimes 2}) &= (st)^2 - 4(mt^2+ns^2-4nm) = (t^2-4n)(s^2-4m) \\
&= d(S)((1 \wedge x)^{\otimes 2})d(T)((1 \wedge y)^{\otimes 2})
\end{aligned}
\end{equation}
The first statement then follows.  For the second, again on affine opens we have the isomorphism
\begin{equation} \label{eqn:twedgehom}
\begin{aligned} 
\tbigwedge^2 S \otimes_R \tbigwedge^2 T &\to \tbigwedge^2 (S*T) \\
(1 \wedge x) \otimes (1 \wedge y) &\mapsto 1 \wedge w,
\end{aligned}
\end{equation}
which glues to give the desired isomorphism globally. 
\end{proof}

\begin{lem} \label{lem:quaddiscsurj}
The discriminant maps 
\begin{center}
$\disc:\scrQuad(X) \to \scrDisc(X)$ and\/ $\disc:\scrQuad(X;\scrO_X) \to \scrDisc(X;\scrO_X)$ 
\end{center}
are surjective homomorphisms of sheaves of commutative monoids.
\end{lem}

\begin{proof}
The fact that these maps are homomorphisms of monoids follows locally from Lemma \ref{lem:homquad}.  We show these maps are surjective locally, and for that we may assume $X=\Spec R$ and $L=Re$.  We refer to Lemma \ref{lem:orbitsdisc} and Example \ref{exm:orbitsdisc}: given any $d \in R$ such that $d=t^2-4n$ with $t,n \in R$, we have the quadratic ring $R[x]/(x^2-tx+n)$ of discriminant $d$.  
\end{proof}

We are now ready to prove Theorem A.

\begin{thm} \label{thma}
Construction \textup{\ref{keycons}} is the unique system of binary operations 
\[ * = *_X : \Quad(X) \times \Quad(X) \to \Quad(X), \]
one for each scheme $X$, such that:
\begin{enumroman}
\item $\Quad(X)$ is a commutative monoid under $*$, with identity element the class of ${\scrO_X \times \scrO_X}$;
\item For each morphism $f:X \to Y$ of schemes, the diagram
\[
\xymatrix{
\Quad(Y) \times \Quad(Y) \ar[r]^(0.63){*_Y} \ar[d] & \Quad(Y) \ar[d]^{f^*} \\
\Quad(X) \times \Quad(X) \ar[r]^(0.63){*_X} & \Quad(X) 
} \]
is commutative; and 
\item If $X=\Spec R$ and $S,T$ are separable quadratic $R$-algebras with standard involutions $\sigma,\tau$, then $S*T$ is the fixed subring of $S \otimes_R T$ under $\sigma \otimes \tau$.
\end{enumroman}
\end{thm}

\begin{proof}
By (\ref{eq:freequadtn}), the universal free quadratic algebra with basis is the algebra 
\[ S\sbuniv=R\sbuniv[x]/(x^2-tx+n) \] 
where $R\sbuniv=\Z[t,n]$ is the polynomial ring in two variables over $\Z$: in other words, for any commutative ring $R$ and free quadratic $R$-algebra $S$ with basis, there is a unique map $f:R\sbuniv \to R$ such that $S=f^* S\sbuniv  = S\sbuniv \otimes_{f,R\sbuniv} R$.  By (ii), then, following this argument on an open affine cover, we see that the monoid structure on $\Quad(\Spec R\sbuniv)$ determines the monoid structure for all schemes $X$.  

Dispensing with subscripts, consider $S=R[x]/(x^2-tx+n)$ and $T = R[y]/(y^2-sy+m)$ where $R=\Z[t,n,s,m]$; we show there is a unique way to define $S * T$.  

To begin, we claim that $S*T$ is free over $R$.  As $R$-modules, we can write $S*T = R \oplus I z$ where $I \subseteq F=\Frac(R)$ is a projective $R$-submodule of $F$ and the class $[I] \in \Pic(R)$ well-defined.  But $\Pic(\Z[t,n,s,m]) \simeq \Pic(\Z)=\{0\}$ ($\Z$ is \emph{seminormal} \cite{GilmerHeitmann}), so $I \simeq R$.

Now let 
\[ D=(t^2-4n)(s^2-4m). \]  
Then $S[1/D]$ and $T[1/D]$ are separable over $R[1/D]$, with involutions $\sigma(x) = t-x$ and $\tau(y)=s-y$.    By (iii), the product $S[1/D] * T[1/D]$ is the subring of $S[1/D] \otimes_{R[1/D]} T[1/D]$ generated by
\[ z = x \otimes y + \sigma(x) \otimes \tau(y) = 2(x \otimes y) - s(x \otimes 1) - t(1 \otimes y). \]
Then
\begin{align*} 
z^2 &= x^2 \otimes y^2 + 2nm + \sigma(x)^2 \otimes \tau(y)^2 \\
&= (tx-n) \otimes (sy-m) + 2nm + (t\sigma(x) - n) \otimes (s\tau(y)-m) \\
&= ts(x \otimes y + \sigma(x) \otimes \tau(y)) - mt((x + \sigma(x)) \otimes 1) 
- ns(1 \otimes (y + \tau(y))) + 4nm \\
&= (st)z - (mt^2 + ns^2 - 4nm).
\end{align*}
In particular, $S[1/D] * T[1/D] \simeq R[1/D] \oplus R[1/D] z$.  

By (ii), $(S * T)[1/D] \simeq S[1/D] * T[1/D]$, and we have $S*T \subseteq (S*T)[1/D]$.  Since $R$ is a UFD and $S*T$ is free over $R$, it is generated as an $R$-algebra by an element of the form $(az+b)/D^k$ for some $a,b \in R$ and $k \in \Z_{\geq 0}$.  But by (\ref{eq:disccalc}), $d((S*T)[1/D]) = D$, so $d(S*T) = (a/D^k)^2 D \in R$, thus $a/D^k \in R$.  Since $\trd((az+b)/D^k) = (ast+2b)/D^k \in R$, we conclude that $2b/D^k \in R$; since $D$ is not divisible by $2$, by Gauss's lemma we have $b/D^k \in R$, so without loss of generality we may take $b=0$ and $S*T$ is generated by $az$ for some $a \in R$.  Since $R^\times = \{\pm 1\}$ and $(S*T)[1/D]$ is generated by $z$, we must have $a=D^k$ for some $k \in \Z_{\geq 0}$.  Finally, we consider 
\[ \frac{\Z[x]}{(x^2-tx+n)} * \frac{\Z[y]}{(y^2-y)} = \frac{\Z[z]}{(z^2 - D^k t z + D^{2k}n)} \]
over $\Z[t,n]$.  The algebra on the right has discriminant $D^{2k}(t^2-4n)$, but by (i), it must be isomorphic to the algebra on the left of discriminant $(t^2-4n)$, so we must have $D^{2k}=1$, so $k=0$.  Therefore $S*T = R \oplus R z$.  
\end{proof}

Having given the monoid structure, we conclude this section by proving Theorem B.

\begin{thm} \label{thmb}
Let $X$ be a scheme.  Then the following diagram of sheaves of commutative monoids is functorial and commutative with exact rows and surjective columns:
\[
\xymatrix{
\scrQuad(X;\scrO_X) \ar[r] \ar[d]^{\disc} & \scrQuad(X) \ar[r]^-{\wedge^2}  \ar[d]^{\disc} & \scrPic(X) \ar[d]^{\wr} \\
\scrDisc(X;\scrO_X) \ar[r] & \scrDisc(X) \ar[r] & \scrPic(X) 
} \]
\end{thm}

\begin{proof}
The exactness of the bottom row follows from Proposition \ref{prop:propexactdiscs}.  The exactness of the top row and commutativity of the diagram follows by the same (trivial) argument.  Surjectivity follows from Lemma \ref{lem:quaddiscsurj}.
\end{proof}








\section{Proof of Theorem C}

In this section, we prove Theorem C and conclude with some final discussion.
 
Let $R$ be a commutative ring and let ${R[4]=\{a \in R : 4a=0\}}$.  Let
\[ \wp(R)=\{r+r^2 : \text{$r \in R$}\}. \]
and let $\wp(R)[4] = \wp(R) \cap R[4]$.  Note that $4(r+r^2)=0$ if and only if $(1+2r)^2=1$, so we have equivalently 
\[ \wp(R)[4]=\{r+r^2 : \text{$r \in R$ and $(1+2r)^2=1$}\}. \] 

\begin{lem} 
$\wp(R)[4]$ is a subgroup of $R[4]$ under addition.
\end{lem}

\begin{proof}
We have $0=0+0^2 \in \wp(R)[4]$.  If $n=r+r^2 \in \wp(R)[4]$ and $m=s+s^2 \in \wp(R)[4]$ then 
\[ (r+s+2rs)+(r+s+2rs)^2 = (r + r^2) + (s+ s^2) + 4(r+r^2)(s+s^2) = n+m \]
and $4(n+m)=0$, so $n+m \in \wp(R)[4]$.  Finally, if $n \in \wp(R)[4]$ then $-n=3n \in \wp(R)[4]$ by the preceding sentence.
\end{proof}

We define the \defi{Artin-Schreier group} $\AS(R)$ to be the quotient 
\[ \AS(R)=\frac{R[4]}{\wp(R)[4]}. \]
Since $2R[4] \subseteq \wp(R)[4]$, the group $\AS(R)$ is an elementary abelian $2$-group.

We define a map $i:\AS(R) \to \Quad(R;R)$ sending the class of $n \in \AS(R)$ to the isomorphism class of the algebra $S=R[x]/(x^2-x+n)$.

\begin{prop} \label{prop:ASiwelldef}
The map $i:\AS(R) \to \Quad(R,R)$ is a (well-defined) injective map of commutative monoids.
\end{prop}

\begin{proof}
Let $S=i(n)=R[x]/(x^2-x+n)$ and $T=i(m)=R[y]/(y^2-y+m)$ with $n,m \in \AS(R)$.  Then $S \simeq T$ if and only if $y=u(x+r)$ for some $u \in R^\times$ and $r \in R$, which by (\ref{eq:changeofbasis}) holds if and only if $u(1+2r)=1$ and $u^2(n+r+r^2)=m$; these are further equivalent to $1+2r \in R^\times$ and 
\[ n+r+r^2=m(1+2r)^2 = (1+4r+4r^2)m. \]
But $4m=0$ so $n+r+r^2=m$ and since $4n=0$ we have $4(r+r^2)=0$.  Thus $S \simeq T$ if and only if $(1+2r)^2=1$ and $n+r+r^2=m$, as desired.  It follows from Construction \ref{keycons} that $S * T = R[w]/(w^2-w+(n+m))$, since $4nm=0$, and $i(0)=R[w]/(w^2-w)$ is the identity, so $i$ is a homomorphism of monoids.
\end{proof}

We now prove Theorem C, and recall $\Quad(R;R)$ is the set of isomorphism classes of free quadratic R-algebras.

\begin{thm} \label{thmc}
Let $R$ be a commutative ring and let $d \in R$ be a discriminant.  Then the fiber $\disc^{-1}(d)$ of the map
\[ \disc:\Quad(R;R) \to \Disc(R;R) \]
above $d$ has a unique action of the group $\AS(R)/\ann_R(d)[4]$ compatible with the inclusion of monoids $\AS(R) \hookrightarrow \Quad(R;R)$.
\end{thm}

\begin{proof}
A (free) quadratic $R$-algebra with basis $1,x$ such that $(x-\sigma(x))^2=d$ is of the form $S=R[x]/(x^2-tx+n)$ with $t^2-4n=d$.  Let $m \in R[4]$.  Then by Construction \ref{keycons} we have
\[ S * (R[y]/(y^2-y+m)) = (R[y]/(y^2-y+m)) * S = R[w]/(w^2-tw + dm+n) \]
since $4m=0$ so $dm=(t^2-4n)m=t^2m$.  Thus we have an action of $R[4]$ on the set of these quadratic $R$-algebras with basis and a free action of $R[4]/\ann_R(d)[4]$.  Two quadratic $R$-algebras $S$ and $S'$ are in the same orbit if and only if $t'=t$ and $n'=dm+n$ for some $m \in R[4]$ if and only if $t'=t$ and $n'-n \in dR[4]$; therefore, the orbits are indexed noncanonically by the set 
\[ \{t \in R : t^2 \equiv d \psmod{4R}\} \times R[4]/dR[4]. \]
We now descend to isomorphism classes: by Proposition (\ref{prop:ASiwelldef}), the monoid multiplication $*$ is well-defined on isomorphism classes, giving the unique action of $\AS(R)/\ann_R(d)[4]$ on the fiber over $d$.  
\end{proof}

Under favorable hypotheses, the action of $\AS(R)$ is free, and so we make the following definition.

\begin{defn}
An element $t \in R$ is \defi{sec} (an abbreviation for \defi{square even cancellative}) if the following two conditions hold:
\begin{enumroman}
\item $t$ is a nonzerodivisor, and
\item $r^2,2r \in tR$ implies $r \in tR$ for all $r \in R$.
\end{enumroman}

A quadratic $R$-algebra $S$ is \defi{sec} if $\disc(S)$ is a nonzerodivisor and there a basis $1,x$ for $S$ such that $\trd(x)$ is sec.  A quadratic $\scrO_X$-algebra $\scrS$ is \defi{sec} if $\scrS$ is sec on an open affine cover of $X$.
\end{defn}

\begin{prop}
Let $R$ be a commutative ring.  Then the action of $\AS(R)$ on the set of sec quadratic $R$-algebras of discriminant $d$ is free.
\end{prop}

\begin{proof}
We continue as in the proof of Theorem \ref{thmc}.  Let $m \in R[4]$ and let $S$ be a sec quadratic $R$-algebras with discriminant $d$.  Let $1,x$ be a basis for $S=R[x]/(x^2-tx+n)$ such that $t=\trd(x)$ is sec.  To show the action is free, suppose that
\begin{equation} \label{eq:SRyS}
S * (R[y]/(y^2-y+m)) = R[w]/(w^2-tw+dm+n) \simeq S;
\end{equation}
we show that $m \in \wp(R)$.  Equation (\ref{eq:SRyS}) holds if and only if there exist $u \in R^\times$ and $r \in R$ such that $t=u(t+2r)$ and $dm+n=u^2(n+tr+r^2)$.  Consequently, $u^2d = d$.  Since $S$ is sec, $d$ is a nonzerodivisor, so $u^2-1=0$, and thus $t^2m=dm=tr+r^2$ so $r^2 = (tm-r)t$.  We also have $2r=(1-u)t$, so since $t$ is sec and $r^2,2r \in tR$, we conclude that $r \in tR$.  Let $r=at$ with $a \in R$.  Then substituting, we have $t^2(a^2+a-m) = r^2+tr-t^2m = 0$.  Since $t$ is a nonzerodivisor, we conclude that $a^2+a-m=0$, so $m \in \wp(R)$ as claimed.  
\end{proof}

\end{document}